\documentclass[a4paper,american,reqno]{amsart}

\newif\ifPreprint \Preprinttrue
\newif\ifSubmission \Submissionfalse

\usepackage{babel}
\usepackage[utf8]{inputenc}
\usepackage[T1]{fontenc}
\usepackage{lmodern}
\usepackage{xspace}
\usepackage{mathtools}
\usepackage{amsmath}
\usepackage{enumitem}
\usepackage{accents}
\usepackage{todonotes}
\usepackage{csquotes}
\usepackage{microtype}
\usepackage[citestyle=alphabetic, 
            bibstyle=alphabetic, 
            firstinits = true,
            maxbibnames = 100,
            maxcitenames = 2,
            isbn = false,
            backend = biber]{biblatex}
\usepackage{hyperref}
\hypersetup{colorlinks,
            citecolor=blue,
            urlcolor=blue,
            linkcolor=blue}
\usepackage{cleveref}

\usepackage[a4paper,top=2cm,left=2cm,right=2cm,bottom=2cm]{geometry}
\usepackage{amssymb}
\usepackage{hyperref}
\usepackage{todonotes}
\usepackage{amsthm, mathtools}

\usepackage[T1]{fontenc}
\usepackage[utf8]{inputenc}
\usepackage{color}
\usepackage{graphicx}
\usepackage{algpseudocode}
\usepackage{algorithm}
\usepackage{cleveref}
\usepackage{bbm}

\makeatletter
\patchcmd{\@settitle}{\uppercasenonmath\@title}{\scshape\large}{}{}
\patchcmd{\@setauthors}{\MakeUppercase}{\scshape\normalsize}{}{}
\makeatother

\vfuzz \hfuzz
\algdef{SE}[SUBALG]{Indent}{EndIndent}{}{\algorithmicend\ }%
\algtext*{Indent}
\algtext*{EndIndent}
\theoremstyle{plain}
\newtheorem{assumption}{Assumption}
\newtheorem{lemma}{Lemma}
\newtheorem{coroll}{Corollary}
\newtheorem{theorem}{Theorem}

\theoremstyle{definition}

\newtheorem{remark}{Remark}

\theoremstyle{remark}
\newcommand{\Bcal}{\mathcal{B}}

\newcommand{\Lcal}{\mathcal{L}}
\newcommand{\Ncal}{\mathcal{N}}

\newcommand{\Pcal}{\mathcal{P}}
\newcommand{\Xcal}{\mathcal{X}}
\newcommand{\Ycal}{\mathcal{Y}}
\renewcommand{\P}{\mathbb{P}}
\newcommand{\E}{\mathbb{E}}
\newcommand{\env}{(u_\text{env})}
\newcommand{\soc}{\text{soc}}

\newcommand{\st}{\text{s.t.}}

\newcommand{\R}{\mathbb{R}}

\usepackage{xcolor}
\usepackage{pgfplots}
\pgfplotsset{width=10cm,compat=1.9}

\usepgfplotslibrary{external}
\bibliography{bibliography}

\begin{document}
\title{A mixed-integer approximation of robust optimization problems with mixed-integer adjustments}

\author[J. Kronqvist, B. Li, J. Rolfes]%
{Jan Kronqvist, Boda Li, Jan Rolfes}

\address[J. Kronqvist, J. Rolfes]{%
  KTH - Royal Institute of Technology, Stockholm,
  Sweden, Department of Mathematics, Lindtstedtsvägen 25, SE-100 44 Stockholm;
  Sweden}
\email{\{jankr,jrolfes\}@kth.se}

\address[B. Li]{%
  ABB Corporate Research Center, Ladenburg, Germany, 
  Wallstadter Str. 59, 68526 Ladenburg}
\email{\{davidlee95th\}@gmail.com}

\date{\today}

\maketitle
%\tableofcontents
%REMOVE IN FINAL PAPER

\begin{abstract}
    In the present article we propose a mixed-integer approximation of adjustable-robust optimization (ARO) problems, that have both, continuous and discrete variables on the lowest level. As these trilevel problems are notoriously hard to solve, we restrict ourselves to weakly-connected instances. Our approach allows us to approximate, and in some cases exactly represent, the trilevel problem as a single-level mixed-integer problem. This allows us to leverage the computational efficiency of state-of-the-art mixed-integer programming solvers. We demonstrate the value of this approach by applying it to the optimization of power systems, particularly to the control of smart converters.
\end{abstract}
\section{Introduction}
Optimization under uncertainty is a very active and growing subfield in mathematical optimization as in many real world applications, the input data is either not known in advance or may be subject to perturbations. In the current literature two major approaches have been developed in order to address this challenge, stochastic optimization (SO), see e.g. \cite{Birge2011a} for further details, and robust optimization (RO), see e.g. \cite{bertsimas2011theory}. In stochastic optimization it is generally assumed that the underlying probability distribution of the uncertain parameters is known, whereas robust optimization methods do not require any knowledge about the distribution but assume that the uncertain data is contained in a predefined set of scenarios. In addition, robustly optimal solutions tend to be rather conservative as one considers the worst-case scenario within the predefined uncertainty set. 

In order to address the conservativeness inherent in robust optimization, adjustable robust optimization (ARO) was introduced in \cite{ben2004adjustable} and grew to a very active subarea in robust optimization. On the one hand, it serves as a natural extension of the concept of RO, where uncertainties are not only addressed a priori by making first stage or here-and-now decisions, but also allowing to react to the realization of the uncertainty by second stage or wait-and-see decisions. Hence, ARO is often used to provide more competitive solutions without sacrificing the uncertainty protection that is provided by RO. This is particularly useful in applications such as gas networks, see e.g. \cite{Assmann2018a} and the seminal unit commitment problem in power systems, see e.g. \cite{Lorca2017a}.

Methodically, AROs are approached either by approximation schemes or by exact reformulations. In \cite{ben2004adjustable} an approximation scheme was presented, where the second stage decisions are restricted to affinely depend on the uncertain parameters. This so-called affine decision rule significantly simplifies the underlying trilevel problem and often leads to tractable optimization problems such as linear or semidefinite programs. Additionally, more elaborate decision rules such as piecewise-linear decision rules, see \cite{bertsimas2015design}, can also address mixed-integer AROs. For further details, we refer to the excellent survey \cite{Yanikoglu2019a}.

Exact reformulations of AROs are usually computationally intractable, particularly if non-convexities such as binary variables are involved. Consequently the literature pivots to decomposition approaches, see \cite{Bienstock2008a}, \cite{zeng2013solving} and \cite{Bertsimas2013a}. Furthermore, in \cite{Avraamidou2019a} a parametric programming approach is given, that solves AROs with a small number of binary variables to global optimality. For larger instances however, parametric programming as of now seems to fall behind decomposition approaches in terms of computational runtime. Hence, the present work aims to contribute to close this computational gap. To this end, we introduce a single-level MIP, that approximates a significant class of AROs and can be proven to be exact under some further assumptions.

As mentioned above, adjustable robust optimization has been active research topic in recent years and it has been found useful in a large variety of areas. Applications dealing with adjustable robustness include, model predictive control \cite{tejeda2019explicit} for dealing with process-model mismatch, process scheduling under uncertainty \cite{lappas2016multi,li2008robust,lin2004new,grossmann2016recent}, multi-task scheduling with imperfect tasks \cite{lappas2019adjustable}, for analysing resiliency and flexibility of chemical process \cite{grossmann2014evolution,zhang2016relation}, to handle uncertainty in supply chains \cite{ben2005retailer,buhayenko2017adjustable}, and portfolio optimization \cite{takeda2008adjustable}. Here, we will not go into more details on these applications, but will elaborate on the application of AROs to find the optimal operating points of smart converters in power systems.

Renewable energy, particularly wind and photovoltaics, has seen rapid growth in recent years \cite{bhandari2015optimization}. These energy sources are widely connected to the power grid, which can greatly reduce the use of fossil fuels and electricity costs. However, this integration also introduces power uncertainty that can threaten the stability and safety of the grid \cite{wei2014robust}.

Smart inverters \cite{zhao2018power} and energy storage technology \cite{weitemeyer2015integration} provide a solution to this problem. The former controls the power output of renewables, while the latter could smooth power fluctuations through charging and discharging. One of the goals of power system optimization is to balance the use of renewable energy and system safety by setting the operating state and set points of inverters and storages.

Historically, power system operation-related optimization involves deterministic optimization, see e.g. \cite{Anjos2017a} for a broader overview on the seminal unit commitment problem and \cite{wang1995short} for an application to security-constrained unit commitment or \cite{sanders1987algorithm} for security-constrained economic dispatch, assuming accurate renewable energy forecasting. However, this stream of research does not account for potential uncertainties in the system, which may lead to infeasible operating states. Addressing these challenges has gained increased attention due to the incorporation of renewables in modern power systems. 

Robust optimization has gained attention for its potential to address this issue and has been extensively studied in power system operation. We refer to Conejo et. al. \cite{Conejo2022a} for a brief survey. Moreover, multi-level robust optimization has also be applied to a various challenges that arise in power system operation, such as robust unit commitment \cite{zhao2013unified, lorca2016multistage} or robust optimal power flow \cite{zhang2013robust} to name a few. It is also applied to the energy management of power systems with smart inverters and energy storage \cite{yu2020energy, lekvan2021robust, yang2021robust}. While these studies provide excellent use cases of multi-level optimization models in power systems, they tend to oversimplify the discrete variables in inverter and storage control models by converting them into continuous variables for increased tractability. However, this may result in infeasible solutions in practical applications since we overestimate the possibilities, the grid operator has at its disposal to readjust the grid. Conversely, the proposed method in this work preserves discrete variables and thus guarantees the feasibility of the computed operating points. Moreover, as our only approximation comes from relaxing the adversarial problem posed by the uncertainties in the system, we provide (potentially overly conservative) operating points, whose feasibility can be guaranteed.

This work is structured as follows. In Section \ref{Sec: problem setting}, the adjustable robust problem formulation is introduced. Moreover, we discuss the key assumption, that these problems are weakly connected and present our main results. Section \ref{Sec: powerflow_theory} introduces the application to smart converters in networks and demonstrates the consequences of the approximation results from Section \ref{Sec: problem setting} in this setting. Subsequently, Section \ref{Sec:computationals_results} illustrates the results and presents numerical evidence, that the presented method outperforms existing algorithms on weakly-connected trilevel problems.

\section{A MIP approach to robust optimization with MIP adjustments}\label{Sec: problem setting}
Adjustable robust programming consist of three separate types or \emph{levels} of variables. The first-level variables, denoted in the present paper by $x\in \Xcal$, describe an initial planning approach, which has to be decided first. Then, an uncertainty affects the outcome of this initial planning. Here, we model this uncertainty by a random vector $h\in\Omega\subseteq \R^I$, which is distributed by a probability distribution $\P\in \Pcal(\Omega)$ on a compact domain $\Omega\subseteq \R^I$. The set $\Pcal(\Omega)$ is called the \emph{ambiguity set} of probability measures and the vector $h$ is referred to as the second-level variable. Lastly, the initial planning $x$ can be adjusted to the uncertainty $h$ by choosing the third-level variables $y\in \Ycal(x,h)$, where in the present article, we suppose that $\Ycal(x,h)\subseteq \R^n\times \{0,1\}^l$ is assumed to be defined through linear constraints. Thus, in total the basic DRO setting can be defined as follows:
\begin{equation}\label{Prob: entire_adjustable_DRO}
    \min_{x\in \Xcal} G(x) + \max_{\P\in \Pcal(\Omega)}\left(\E_\P \left(\min_{y\in \Ycal(x,h)} c^\top y\right)\right),
\end{equation}
where $\Xcal$ and $\Omega$ are assumed to be a polytopes and $\mathcal{Y}(x,h)$ a polytope intersected with an integer lattice. However, in the present article, we restrict ourselves to a standard robust ambiguity set, i.e., for the polytope $\Omega$ the ambiguity set is defined as a set of Dirac measures $\Pcal(\Omega)\coloneqq \{\delta_{\{h\}}: h\in \Omega \}$ and aim to solve instances of
\begin{equation}\label{Prob: entire_adjustable_RO}
    \min_{x\in \Xcal} G(x) + \max_{h \in \Omega}\min_{y\in \Ycal(x,h)} c^\top y.
\end{equation}
Note, that \eqref{Prob: entire_adjustable_RO} is still considered to be a very challenging problem as it contains the $\mathrm{NP}$-complete MIP $\min_{y\in \Ycal(x,h)} c^\top y$ as a subproblem, see Section 6 in \cite{Yanikoglu2019a} for further details. In the present article, we focus on turning the above trilevel-problem into a single-level MIP. To this end, we fix the integral part of $y$ and dualize the resulting LP in order to achieve a bilevel problem with a bilinear objective. This objective is then relaxed with McCormick envelopes and dualized again in order to achieve a single-level LP for every fixed integral assignment of $y$. However, instead of evaluating the resulting LP relaxations for every integer combination, we can simply reincorporate the integrality of $y$ as an additional constraint and, thus, obtain a single-level MIP. This approach may lead to overly pessimistic outcomes as using the McCormick relaxations strengthens the adverserial in the multilevel problem, but can be shown to be exact under some assumptions on $h$ and the dual variables of $\min_{y\in \Ycal(x,h)} c^\top y$. 

However, this method may not work on strongly connected multilevel problems since broadly speaking the hardness in multilevel optimization problems comes from strong connections between the levels, we consider the following subclass of adjustable robust problems:

We define a problem of type \eqref{Prob: entire_adjustable_RO} \emph{weakly connected}, if the only relation between the first-,second- and third level variables is a linear relation in $\Ycal(x,h)$:
$$\Ycal(x,h) =\{y\in \R^n\times \{0,1\}^l: A'y\geq b', B y \geq B_x x + B_h h + b_0\},$$
i.e., the first-level variables $x$ and the second-level random vector only affect the right-hand side of the constraints defining $\Ycal(x,h)$. Weakly connected adjustable robust problems include among others, problems with affine decision rules ($B^\top = (I_n,-I_n), B_x=0, B_h^\top = (Q^\top,-Q^\top), b_0 = (q^\top,-q^\top)$, where $Q,q$ can be chosen arbitrarily.

Moreover, in order to work with the levels separately, we denoted by $A'y\geq b'$ the constraints, that solely deal with third-level variables, i.e., that are neither affected by the first-level variables $x$ nor by the second-level variables $h$ and by $By \geq B_x x + B_h h + b_0$
the constraints, that are affected by the upper levels. Furthermore, instances that are particularly well suited for our approach should have relatively few non-zero entries in the matrix $B_h$. Few non-zero entries in $B_h$ is not a strict requirement, but both the approximation quality and computational complexity can scale with the number of non-zero elements. 

In addition, let us consider an LP inner-approximation of the third-level program $\min_{y\in \Ycal(x,h)} c^\top y$, i.e. we fix the integer variables $y_{n+1},\ldots , y_{n+l}\in \{0,1\}$ by a set of linear constraints denoted by $A_f y \geq b_f$. Consequently, we consider the following LP:
\begin{subequations}\label{Prob: third_level_primal_general}
    \begin{align}
        \min\ & c^\top y \\
        \st\ & A'y \geq b',\label{Constr: third_level_non-entangled1}\\
        & A_f y \geq b_f \label{Constr: third_level_non-entangled2}\\
        & By\geq B_x x + B_h h + b_0\label{Constr: third_level_entangled}
    \end{align}
\end{subequations}
Let $A^\top = ((A')^\top,A_f^\top), b^\top = ((b')^\top,b_f^\top)$ and $\alpha$ denote the dual variables that correspond to Constraints \eqref{Constr: third_level_non-entangled1} and \eqref{Constr: third_level_non-entangled2}, i.e. to $Ay\geq b$. Additionally, we denote by $\beta$ the dual variables corresponding to \eqref{Constr: third_level_entangled} and obtain as the dual program of \eqref{Prob: third_level_primal_general}:
\begin{subequations}\label{Prob: third_level_dual_general}
    \begin{align}
        \max\ & b^\top \alpha  + (B_x x + B_h h + b_0)^\top \beta\\
        \st\ & \begin{pmatrix}A^\top & B^\top \end{pmatrix}\begin{pmatrix}\alpha\\ \beta\end{pmatrix} = c,\\
        & \alpha,\beta\geq 0,
    \end{align}
\end{subequations}

We observe, that on the second level the objective decomposes into a bilinear part $(B_h h)^\top \beta = h^\top B_h \beta$ and a linear one $b^\top \alpha  + (B_x x + b_0)^\top \beta$. Moreover, the amount of nonzero entries of $B_h$ crucially determines the number of bilinear terms and thereby the difficulty of computing \eqref{Prob: third_level_dual_general}.

Let $I$ and $J$ denote the index sets of $h$ and $\beta$ respectively. Then, in order to relax the bilinear term $h^\top B_h\beta$, we make the following key assumption:
\begin{assumption}\label{Ass: bounds_on_h_and_beta}
    Both, the dual variables $\beta\in \R^J$ as well as the second-level variables $h\in \R^I$ are bounded, i.e., there are $\beta^-,\beta^+$ and $h^-,h^+$ such that adding
$$h^- \leq h\leq h^+ \text{ and } \beta^-\leq \beta \leq \beta^+$$
to $\Omega$ or \eqref{Prob: third_level_dual_general} does not affect the outcome of \eqref{Prob: entire_adjustable_DRO}.
\end{assumption}
Note that a compact $\Omega$ already implies $h^-\leq h \leq h^+$. Hence, in those cases it suffices to check Assumption \ref{Ass: bounds_on_h_and_beta} for $\beta$. Later, we show that finite, and meaningful, bounds $\beta^-$ and $\beta^+$ can be easily determined for the considered application of finding optimal operating points in power systems. In addition, we mention, that every sharpening of the bounds in Assumption \ref{Ass: bounds_on_h_and_beta} improves the quality of the upcoming results. This is due to the fact, that the boundedness of both, $\beta$ and $h$ enables us to relax the second-level problem $\max_{h\in \Omega} \min_{y\in \Ycal(x,h)} c^\top y$  further using McCormick envelopes for the bilinear terms. In particular, this idea gives rise to the following theorem:

\begin{theorem}\label{Thm: linear_single_level_approximation_general}
    Let the ambiguity set $\Omega$ be a polytope defined by $\Omega=\{h\in \R^I: A_\Omega^\top h +B_\Omega^\top \eta =b_\Omega,\ \eta \geq 0\}$ with $b_\Omega\in \R^k$, i.e. $\Omega$ is compact and $\eta \geq 0$ denote potential nonnegative slack variables in the rows. Let further $\beta^-,\beta^+$ be a lower/upper bound for $\beta$. Then, the following linear program provides an upper bound to \eqref{Prob: entire_adjustable_RO}:
    \begin{subequations}\label{Prob: McCormick_relaxation_three_level_program}
    \begin{align}
        \min\ & G(x) + (\beta^+)^\top u_\beta^+ + (\beta^-)^\top u_\beta^- + b_\Omega^\top u_\Omega + c^\top y\\
        & - \sum_{i\in I, j\in J} \left(h_i^- \beta_j^- \env_{ij}^1  + h_i^+\beta_j^+ \env_{ij}^2 + h_i^+\beta_j^- \env_{ij}^3 + h_i^-\beta_j^+ \env_{ij}^4\right) \notag\\
        \st\ & Ay \geq b,\\
        & By + u_\beta^+ + u_\beta^- -h_i^-\env_{i}^1 - h_i^+\env_{i}^2 - h_i^+ \env_{i}^3 -h_i^-\env_{i}^4 \geq B_x x + b_0 && \text{for every } i\in I\\
        & u_\beta^+ \geq 0,\ -u_\beta^- \geq 0,\\
        & (A_\Omega u_\Omega)_i - \sum_{j\in J} \left( \beta_j^-\env_{ij}^1 + \beta_j^+\env_{ij}^2 + \beta_j^-\env_{ij}^3+\beta_j^+\env_{ij}^4\right) \geq 0 && \text{for every } i\in I,\\
        & B_\Omega u_\Omega \geq 0,\\
        & \env_{ij}^1 + \env_{ij}^2 + \env_{ij}^3 + \env_{ij}^4 \geq (B_h)_{ij} && \text{for every } i\in I, j\in J,\\
        & -\env_{ij}^1, -\env_{ij}^2 \geq 0 && \text{for every } i\in I, j\in J,\\
        & \env_{ij}^3, \env_{ij}^4 \geq 0 && \text{for every } i\in I, j\in J,\\
        & x\in \Xcal.
    \end{align}
\end{subequations}
\end{theorem}

\begin{proof}
    We observe that with Assumption \ref{Ass: bounds_on_h_and_beta} the second-level $\max_{h\in \Omega} \min \{c^\top y: Ay\geq b, By \geq B_x x + B_h h + b_0\}$ can be written as
\begin{subequations}
    \begin{align}
        \max\ & b^\top \alpha + (B_x x + b_0)^\top \beta + \sum_{i\in I, j\in J} (B_h)_{ij} h_i\beta_j \label{Obj: second-level_with_beta_bounds}\\
        \st\ & A^\top \alpha + B^\top \beta = c,\\
        & \beta + \gamma = \beta^+,\\
        & \beta - \delta = \beta^-,\\
        & A_\Omega^\top h + B_\Omega^\top \eta = b_\Omega,\\
        & \alpha,\beta,\gamma,\delta,h,\eta \geq 0,
    \end{align}
\end{subequations}
where we assumed w.l.o.g. that $h^-= 0$ since otherwise, we could substitute $h$ by $h-h^-$ and adjust $b_0$ and $b_\Omega$ accordingly.
Next, we substitute $\kappa_{ij}\coloneqq h_i\beta_j$ in the objective and relax the resulting constraint $\kappa_{ij}= h_i\beta_j$ by a McCormick envelope. Note, that w.l.o.g. $\kappa\geq 0$. If we further introduce suitable nonnegative slack variables $\rho\geq 0$, we obtain the following LP: 
\begin{subequations}\label{Prob: second_level_primal_relaxed_general}
    \begin{align}
        \max\ & b^\top \alpha + (B_x x + b_0)^\top \beta + \sum_{i\in I, j\in J} (B_h)_{ij} \kappa_{ij}\\
        \st\ & A^\top \alpha + B^\top \beta = c,\\
        & \beta + \gamma = \beta^+,\\
        & \beta - \delta = \beta^-,\\
        & A_\Omega^\top h + B_\Omega^\top \eta = b_\Omega,\\
        & \kappa_{ij} = h_i^-\beta_j + h_i\beta_j^- - h_i^-\beta_j^- +\rho_{ij}^1 && \text{for every } i\in I, j\in J \\
        & \kappa_{ij} = h_i^+\beta_j + h_i\beta_j^+ - h_i^+\beta_j^+ +\rho_{ij}^2 && \text{for every } i\in I, j\in J,\\
        & \kappa_{ij} = h_i^+\beta_j + h_i\beta_j^- - h_i^+\beta_j^- -\rho_{ij}^3 && \text{for every } i\in I, j\in J,\\
         & \kappa_{ij} =  h_i\beta_j^+ + h_i^-\beta_j - h_i^-\beta_j^+ -\rho_{ij}^4 && \text{for every } i\in I, j\in J,\\
        & \alpha,\beta,\gamma, \delta,h,\eta,\kappa,\rho \geq 0.
    \end{align}
\end{subequations}

If we denote the dual variables of \eqref{Prob: second_level_primal_relaxed_general} by $y, u_\beta^+, u_\beta^-, u_\Omega, \env_{ij}^1,\env_{ij}^2,\env_{ij}^3,\env_{ij}^4$ respectively, then the result follows by strong duality and including the first-level variables and objectives.
\end{proof}

We observe, that Theorem \ref{Thm: linear_single_level_approximation_general} provides an LP inner approximation of \eqref{Prob: entire_adjustable_RO} in the sense that the adversarial is overestimated, whereas the space of decision variables is underestimated, potentially leading to a more conservative solution. In particular, we simply fixed the discrete or w.l.o.g. binary decisions in $\Ycal(x,h)$ denoted by $y_{m+1},\ldots y_{m+l}\in \{0,1\}$. To this end, we briefly illustrate the key idea behind the following Theorem \ref{Thm: single_level_MIP_general}, that addresses integralities:

Observe, that solving \eqref{Prob: entire_adjustable_RO} is equivalent to solving exponentially many trilevel LPs -- one LP for every fixed choice of $y_{m+1},\ldots y_{m+l}$. This is due to the fact that the proof of Theorem \ref{Thm: linear_single_level_approximation_general} does not depend on the fixing of $(y_{m+1},\ldots , y_{m+l})^\top$ to a vector in $y'\in \{0,1\}^l$. Consequently, we obtain $2^l$ single-level LPs, where each of these LPs yields an inner-approximation of the trilevel LPs given by the fixing of $y_{m+1},\ldots , y_{m+l}$. Finally, instead of solving each of those exponentially many LPs separately, the following MIP gives the same result.

\begin{theorem}\label{Thm: single_level_MIP_general}
    Let the ambiguity set $\Omega$ be a polytope defined by $\Omega=\{h\in \R^I: A_\Omega^\top h +B_\Omega^\top \eta =b_\Omega,\ \eta \geq 0\}$ with $b_\Omega\in \R^k$, i.e. $\Omega\subseteq \R^I$ is compact and $\eta\geq 0$ denote potential nonnegative slack variables in the rows. Let further $\beta^-,\beta^+$ be a lower/upper bound for $\beta$. Then, the following MIP provides an upper bound to the tri-level MIP \eqref{Prob: entire_adjustable_RO}:
    \begin{subequations}\label{Prob: single_level_MIP_general}
    \begin{align}
        \min\ & G(x) + (\beta^+)^\top u_\beta^+ + (\beta^-)^\top u_\beta^- + b_\Omega^\top u_\Omega  + c^\top y \label{Obj: single_level_MIP_general}\\
        & - \sum_{i,j\in [m]} \left(h_i^- \beta_j^- \env_{ij}^1  + h_i^+\beta_j^+ \env_{ij}^2 + h_i^+\beta_j^- \env_{ij}^3 + h_i^-\beta_j^+ \env_{ij}^4\right)\notag\\
        \st\ & Ay \geq b, \label{Constr: single_level_MIP_general_C1}\\
        & By + u_\beta^+ + u_\beta^- -h_i^-\env_{i}^1 - h_i^+\env_{i}^2 - h_i^+ \env_{i}^3 -h_i^-\env_{i}^4 \geq B_x x + b_0 && \text{for every } i\in [m]\\
        & u_\beta^+ \geq 0,\ -u_\beta^- \geq 0,\\
        & (A_\Omega u_\Omega)_i - \sum_{j\in [m]} \left( \beta_j^-\env_{ij}^1 + \beta_j^+\env_{ij}^2 + \beta_j^-\env_{ij}^3 + \beta_j^+\env_{ij}^4\right) \geq 0 && \text{for every } i\in [m],\\
        & B_\Omega u_\Omega \geq 0,\\
        & \env_{ij}^1 + \env_{ij}^2 + \env_{ij}^3 + \env_{ij}^4 \geq (B_h)_{ij} && \text{for every } i,j\in [m],\\
        & -\env_{ij}^1, -\env_{ij}^2 \geq 0 && \text{for every } i,j\in [m],\\
        & \env_{ij}^3, \env_{ij}^4 \geq 0 && \text{for every } i,j\in [m],\\
        & x\in \Xcal, u_\beta^+, u_\beta^- \in \R^J, u_\Omega\in \R^k, \env^1, \env^2, \env^3, \env^4 \in \R^{I\times J}, y\in \R^{n+l},\label{Constr: single_level_MIP_general_env2_lb}\\
        & y_{m+1},\ldots , y_{m+l}\in \{0,1\}. \label{Constr: IP_constraint_general}
    \end{align}
\end{subequations}
\end{theorem}

\begin{proof}
   % Let $x, u_\beta^+, u_\beta^-, u_\Omega, \env^1, \env^2, \env^3, \env^4, y$ be a feasible solution to \eqref{Prob: single_level_MIP_general}. 
   First, we fix $(y_{m+1},\ldots , y_{m+l})^\top =y'$ with a set of linear (in-)equalities $A_fy \geq b_f$. Here, $y'\in \{0,1\}^l$ denotes an arbitrary integer assignment. Suppose we replace \eqref{Constr: IP_constraint_general} by this system of inequalities and denote the resulting feasible set by $\Xcal(y')$, i.e.
    $$\Xcal(y')=\left\{x, u_\beta^+, u_\beta^-, u_\Omega, \env^1, \env^2, \env^3, \env^4, y : \eqref{Constr: single_level_MIP_general_C1} - \eqref{Constr: single_level_MIP_general_env2_lb}, (y_{m+1},\ldots , y_{m+l})^\top = y'\right\}.$$
    
    We further denote the corresponding objective formed by \eqref{Obj: single_level_MIP_general} by $c(x)$, i.e. the LP $\min_{x\in \Xcal(y')} c(x)$ describes \eqref{Prob: single_level_MIP_general} with fixed values of $(y_{m+1},\ldots , y_{m+l})^\top$. 
    %We observe that the inequalities solely depend on $y'$, i.e. are neither connected to the first nor the second level variables $x,h$. Thus, $(y_{m+1},\ldots , y_{m+l})^\top=y'$ behaves in the same manner as Constraint \eqref{Constr: single_level_MIP_general_C1}. 
    Hence, Theorem \ref{Thm: linear_single_level_approximation_general} implies, that $\min_{x\in \Xcal(y')} c(x)$ inner-approximates the following trilevel LP 
    $$\min_{x\in \Xcal} G(x) + \max_{h\in \Omega} \min_{y\in \Ycal(x,h,y')} c^\top y,$$
    where $$\Ycal(x,h,y')=\left\{ y\in \Ycal(x,h): (y_{m+1},\ldots , y_{m+l})^\top=y'\right\}.$$
    In particular, we obtain by Theorem \ref{Thm: linear_single_level_approximation_general} that $\min_{x\in \Xcal(y')} c(x) \geq \min_{x\in \Xcal} G(x) + \max_{h\in \Omega} \min_{y\in \Ycal(x,h,y')} c^\top y$ for every $y'\in \{0,1\}^l$. 
    Subsequently, we conclude
    \begin{align*}
        \eqref{Prob: single_level_MIP_general} & = \min_{y'\in \{0,1\}^l} \min_{x\in \Xcal(y')} c(x) \\
        & \geq \min_{y'\in \{0,1\}^l} \min_{x\in \Xcal} G(x) + \max_{h\in \Omega} \min_{y\in \Ycal(x,h,y')} c^\top y\\
        & = \min_{x\in \Xcal} G(x) + \min_{y'\in \{0,1\}^l}\max_{h\in \Omega} \min_{y\in \Ycal(x,h,y')} c^\top y \\
        & \overset{(*)}{\geq} \min_{x\in \Xcal} G(x) + \max_{h\in \Omega} \min_{y' \in \{0,1\}^l,y\in \Ycal(x,h,y')} c^\top y = \eqref{Prob: entire_adjustable_DRO}. 
    \end{align*}
    As stated above, the first inequality is based on Theorem \ref{Thm: linear_single_level_approximation_general} whereas the second one is a consequence of the max-min-inequality. In the remainder of this proof, we will argue, that $(*)$ is even sharp and thus, any potential differences between the relaxation \eqref{Prob: single_level_MIP_general} and the original problem \eqref{Prob: entire_adjustable_DRO} solely originate from the McCormick envelopes. To this end,
    let us consider the Lagrangian relaxation with penalty terms instead of hard constraints:
    $$f_x(h,y')\coloneqq \min_{y\in \Ycal(x)} c^\top y + u_0^\top(A_\Omega^\top h +B_\Omega^\top \eta -b_\Omega) + u_1^\top((y_{m+1},\ldots , y_{m+l})^\top-y'),$$
    where $\Ycal(x)\supseteq \Ycal(x,h)$ is defined as the relaxation of $\Ycal(x,h)$ occuring if Constraint $A_\Omega^\top h +B_\Omega^\top \eta = b_\Omega$ is dropped. Observe that since $\Ycal(x)$ is bounded, for a given $h\in \Omega$ and $y'\in \{0,1\}^l$, according to Theorem 21 in \cite{Lemarchal2001a}, there exist $u_0\in \R^k,u_1\in \R^l$ that satisfy
    $$\min_{y\in \Ycal(x,h,y')} c^\top y =  f_x(h,y').$$
    Then, since $f_x(h,y')$ is concave in $h$, i.e. on the nonempty set $\Omega$ and further is convex in $y'$ on the compact nonempty set $\{0,1\}^l$, we apply the Ky-Fan theorem \cite{Fan1953a} and conclude equality in $(*)$. 
\end{proof}

We note, that the approximation quality given by Theorem \ref{Thm: single_level_MIP_general} relies solely on the strenght of the relaxation due to the McCormick envelopes. Moreover, as the number of nonzero entries in $B_h$ determines the number of bilinear terms in the second-level, it crucially affects the approximation quality given by the McCormick relaxation. In particular, a large number of bilinear terms may lead to an overly conservative solution. Thus, exploiting problem specific information that improves the bounds $\beta^-,\beta^+,h^-,h^+$ directly increases the solution quality. Moreover, more elaborate approximations of $h^\top B_h \beta$ may lead to a direct improvement of Theorem \ref{Thm: single_level_MIP_general} and are subject to future research. 
\section{Application to smart converters in power systems networks}\label{Sec: powerflow_theory}
In the present section, we apply our results to the questions of how to operate a power grid integrated with smart inverters and storage.

Usually, the optimal operation problem of the power system can be divided into two phases.
\begin{itemize}
    \item The first phase is about day-ahead scheduling. In this phase, the power demand of the power grid needs to be determined and reported to the day-ahead electricity market. In addition, internal conventional generators need to make a power generation plan for the next day to reasonably allocate diesel consumption.
    \item The second phase is about intra-day system operation. Given the day-ahead decisions, operators need to further set operating points and working status of smart inverters and energy storage systems to realize a real-time power balance and reduce the waste of renewables.
\end{itemize}

As can be seen from our analysis, the day-ahead decisions influence the intra-day decisions. The renewable-led uncertainties will further affect the quality of decision-making. In order to realize the safe and economical operation of the power grid, we need to overcome the following three difficulties:
\begin{enumerate}
    \item \textbf{Sequential}: The formulation of the day-ahead strategy is subject to uncertainty realizations and intra-day operations. A multi-layer model is needed to describe the decision sequence as well as uncertainty realizations in the real world.
    
    \item \textbf{Uncertain}: The uncertainties threaten the safe operation of power systems. Power fluctuations induced by renewables may deteriorate power quality and increase electricity costs. Hence, the operation strategy should be robust enough to handle different situations. 

    \item \textbf{Discrete}: The intra-day operation in the second phase often involves many state-switching operations. The states can be modeled as integer variables in the optimization problem. However, these variables destroy the convexity of the model. A new solution method should be developed in such a case.
    
\end{enumerate}

The power system operation problem can be solved using the robust optimization approach with MIP adjustments proposed in Section 2. A multi-layer model is built to describe the sequential decision-making process; the adaptive robust programming is employed to account for worst-case scenarios; the MIP adjustment is then applied to consider the state switching in the system.

At first, we formulate the mathematical model of the power system operation. We mainly follow the notation by
\cite{Yang2019a} but use the simpler DC approximation of Kirchhoff laws in order to model the power flow when operating the grid. 

Let $\Bcal$ denote a set of buses and $\Lcal$ denote a set of lines/branches in a power grid. Additionally, power is generated within the grid either by a set of conventional (fossil fuel) generators denoted by $\Ncal_G$ or by a set of distributed (renewable) generators denoted by $\Ncal_{DG}$. Moreover, the transmission system operator (TSO) may decide to store or release power via a set $\Ncal_S$ of storages, e.g. batteries. The last potential sources and sinks of power is a trading node with other connected regional transmission grids. Here, the TSO may purchase power on the day-ahead market or intra-day. As the day-ahead market is often called first-level market, we denote the amount of power traded day-ahead by $P_{fl}$ and the corresponding market price by $p_{fl}$. Similarly, the amount of power traded intra-day is denoted by $P_{sl}$ and its market price by $p_{sl}$. Note, that a positive value for $P_{fl},P_{sl}$ is interpreted as a purchase and negative values for $P_{fl},P_{sl}$ correspond to a sell of energy.

For the SO's initial planning, one considers the first-level variables $x=(P_G^\top,P_{fl})^\top\in \R^{\Ncal_G}\times \R$ combined with the estimated renewable energy production $P_{DG, forecast}\in \R^{\Ncal_{DG}}$ and ensures that a given total demand $\sum_{i\in \Bcal} P_{d_i}$ is met. Hence, 
$$\Xcal=\{x\in \R^n:\ \eqref{Constr: fl1_define_Pfl}\ \&\ \eqref{Constr: fl2_bounds_P_G} \},$$
where
\begin{subequations}
    \begin{align}
    & P_{fl}^t = \sum_{i\in \Bcal} P_{d_i}^t - \sum_{i\in \Ncal_G}P_{G_i}^t - \sum_{i\in \Ncal_{DG}} P_{DG_i, forecast}^t & \text{for every } t\in T, \label{Constr: fl1_define_Pfl}\\
    & P_{G_i,\text{min}}^t \leq P_{G_i}^t \leq P_{G_i,\text{max}}^t & \text{for every } i\in \Ncal_G, t\in T.\label{Constr: fl2_bounds_P_G}
\end{align}
\end{subequations}
with given parameters $P_{G,\min}, P_{G,\max}\in \R^{\Ncal_G \times T}$ and the subsequent constraints \eqref{Constr: fl1_define_Pfl} and \eqref{Constr: fl2_bounds_P_G}. We like emphasize that the \emph{market-clearing condition} \eqref{Constr: fl1_define_Pfl} ensures the active power balance in the whole system. Moreover, the objective of the first-level is given by 
$$G(x)=\sum_{i\in \Ncal_G,t\in T} c_{G_i^t,2} (P_{G_i}^t)^2 + c_{G_i^t,1} P_{G_i}^t + c_{G_i^t,0} + \sum_{t\in T} p_{fl}^t P_{fl}^t,$$
where $c_{G_i^t,2},c_{G_i^t,1},c_{G_i^t,0}\in \R$ are given generator cost parameters.

Since the uncertainties will impact the initial planning, we will consider the uncertain capacity of the renewable generators, i.e. we denote the second-level variable by $h=P_{DG, \text{max}}$. As these generators are dependent on weather conditions, which are highly uncertain, this is one of the most common uncertainties faced by modern power grids with high proportion of renewable energies \cite{pfenninger2014energy, impram2020challenges}. As the grid stability is crucial, we will address this uncertainty in a robust manner and set the domain of the second-level variables $P_{DG,\max}$ to

\begin{equation}\label{Eq: Def_Omega}
    \Omega=\left\{P_{DG,\text{max}}\in \R^{\Ncal_{DG}\times T}:\ 0\leq  P_{DG_i, \text{max}}^t \leq P_i^+ \forall\ i\in \Ncal_{DG}, t\in T, \sum_{i\in \Ncal_{DG}} P_{DG_i, \text{max}}^t \geq R\sum_{i\in \Ncal_{DG}} P_{DG_i, forecast}^t \right\},
\end{equation}
where $P_i^+$ denotes the technical limit of the renewable generator, i.e. its capacity under optimal conditions and $R\in [0,1]$ denotes a maximal forecast error. 

On the third level, the TSO is able to react to this uncertainty and adjust the initial planning accordingly. In particular, instead of producing $P_G$, the TSO might regulate the energy output to $P_{G,\text{reg}}$ by either increasing the production by adding $P_G^+\geq 0$ or decreasing the production by adding $P_G^-\leq 0$. However, this can only be done at a cost $r^+$ or $r^-$ respectively. Similarly, $P_{DG}$ is the adjusted energy production that deviates from its forecast by $P_{DG}^+$ or $P_{DG}^-$, where deviations are penalized by $f^+,f^-$ respectively. Despite of these regulations, the TSO might trade power intra-day ($P_{sl}$) or decide whether ($\mu_{ch},\mu_{dch}\in \{0,1\}^{\Ncal_S}$) and by how much $P_{ch},P_{dch}$ to charge or discharge the storages for balancing potential power deficiency and surplus. The state of charge of a storage is denoted by $\soc$. Lastly, the power on a line $(k,l)\in \Lcal$ is denoted by $p_{kl}$ and the phase angles of the system by $\vartheta$. Hence, the TSO can adjust the vector $y=(P_{G_i,\text{reg}}, P_{G_i}^+, P_{G_i}^-, P_{DG_i}, P_{DG_i}^+, P_{DG_i}^-, P_{sl}, P_{ch_i}, P_{dch_i},p_{kl}, \theta, \soc, \mu_{ch}, \mu_{dch} )^\top$ in order to satisfy
$$y\in \Ycal(x,h)\coloneqq \left\{y\in \R^m:\ \eqref{Constr: adjust_fossil_fuel_generators}-\eqref{Constr: mu_bound}\right\},$$
where the constraints \eqref{Constr: adjust_fossil_fuel_generators}-\eqref{Constr: mu_bound} are given below:

\begin{enumerate}
    \item First, we consider the following \emph{Generator and DG output constraints}:
    \begin{subequations}
        \begin{align}
            & P_{G_i,\text{reg}}^t = P_{G_i}^t + P_{G_i}^{t,+}+ P_{G_i}^{t,-} & \text{for every } i\in \Ncal_G, t\in T, \label{Constr: adjust_fossil_fuel_generators}\\
            & P_{G_i,\text{min}}^t \leq P_{G_i,\text{reg}}^t \leq P_{G_i,\text{max}}^t & \text{for every } i\in \Ncal_G, t\in T,\\
            & P_{DG_i}^t = P_{DG_i,\text{forecast}}^t + P_{DG_i}^{t,+}+P_{DG_i}^{t,-} & \text{for every } i\in \Ncal_{DG}, t\in T, \label{Constr: adjust_distributed_generators}\\
            & P_{DG_i,\text{min}}^t \leq P_{DG_i}^t \leq P_{DG_i, \text{max}}^t & \text{for every } i\in \Ncal_{DG}, t\in T, \label{Constr: P_DG_leq_P_DG_max}\\
            & P_{sl}^t-P_{fl}^t=-\mathbbm{1}^\top (P_{G}^{t,+} +P_G^{t,-}) -\mathbbm{1}^\top (P_{DG}^{t,+}+P_{DG}^{t,-}) -\mathbbm{1}^\top (P_{dch}^t-P_{ch}^t) & \text{for every } t\in T.\label{Constr: second_level_market_clearing}
        \end{align}
    \end{subequations}
    Constraints \eqref{Constr: adjust_fossil_fuel_generators}-\eqref{Constr: P_DG_leq_P_DG_max} describe the output range of the conventional and renewable generators and the potential impact of uncertainties. In particular, \eqref{Constr: P_DG_leq_P_DG_max} shows that the third-level variables $P_{DG_i}^t$ are restricted by the uncertainties realized in the second level ($P_{DG_i, \text{max}}^{t}$). Constraint \eqref{Constr: second_level_market_clearing} illustrates the market clearing on the intra-day market, which reflect the actual power demand-supply relations.
    
    \item Second, we consider the \emph{operation constraints}:
    \begin{subequations}
        \begin{align}
            & 0 \leq P_{G_i}^{t,+} \leq P_{G_i, \text{max}}^{t,+} & \text{for every } i\in \Ncal_G, t\in T, \label{Constr: pg1}\\
            & P_{G_i, \text{min}}^{t,-} \leq P_{G_i}^{t,-} \leq 0 & \text{for every } i\in \Ncal_G, t\in T,\\
            & P_{DG_i}^{t,+}\geq 0 & \text{for every } i\in \Ncal_{DG}, t\in T,\\
            & P_{DG_i}^{t,-} \leq 0 & \text{for every } i\in \Ncal_{DG}, t\in T. \label{Constr: pg4}
        \end{align}
    \end{subequations} 
    Constraints \eqref{Constr: pg1} -\eqref{Constr: pg4} limit the real output derivations of conventional and renewable generators.    
    \item Third, we consider the \emph{power flow constraints} with a DC approximation for the given, constant line reactance ($x_{ij}>0$) and demand in active power ($P_k^{d,t}$) at every time step $t$ and bus $k$. The nodal power flow balance is established in \eqref{Constr: power_flow1} and separately in \eqref{Constr: power_flow2} for the root node. The branch power flow is established in \eqref{Constr:branch_power_balance}. The constraints are listed as follows:
    \begin{subequations}
    \begin{align}
        & \sum_{i\in \Ncal_G: i\sim k} P_{G_i,\text{reg}}^t +\sum_{i\in \Ncal_{DG}: i\sim k} P_{DG_i}^t + \sum_{i\in \Ncal_S: i\sim k} (P_{dch_i}^t-P_{ch_i}^t) - P_k^{d,t} = \sum_{l \in \delta(k)} p_{kl}^t && \forall k\in \Bcal\setminus\{0\}, t\in T, \label{Constr: power_flow1}\\
        & \sum_{i\in \Ncal_G: i\sim 0} P_{G_i,\text{reg}}^t +\sum_{i\in \Ncal_{DG}: i\sim 0} P_{DG_i}^t + \sum_{i\in \Ncal_S: i\sim 0} (P_{dch_i}^t-P_{ch_i}^t) +P_{sl}^t - P_0^{d,t} = \sum_{l \in \delta(k)} p_{0l}^t && \forall t\in T, \label{Constr: power_flow2}\\
        & p_{ij}^t = \frac{1}{x_{ij}} (\theta_i^t-\theta_j^t) && \forall \{i,j\}\in \Lcal, t\in T \label{Constr:branch_power_balance}
    \end{align}
    \end{subequations}
    \item Fourth, to guarantee the safe operation of the branch, the power flow should not exceed the branch's capacities. Hence, we consider the \emph{branch thermal constraints}:
    \begin{subequations}
        \begin{align}
            & -s_{ij, \max} \leq p_{ij}^t \leq s_{ij, \max} & \text{for every } \{i,j\}\in \Lcal, t\in T.
        \end{align}
    \end{subequations}
    \item Fifth, we consider the \emph{storage constraints}. Noticing that the storage operation involves two actions, the action shifts need to be considered. Two binary variables $\mu_{ch}, \mu_{dch}$ are defined to represent the storage action, where $\mu_{ch}, \mu_{dch}\in \{0,1\}^{\Ncal_S\times T}$. Moreover, only one of those variables can be 1 at time $t$. Thus, we have the following constraints:
    \begin{subequations}
        \begin{align}
        & \soc_{i,\text{min}}^t \leq \soc_i^t \leq \soc_{i,\text{max}}^t & \text{for every } i\in \Ncal_S, t\in T, \label{Constr:soc1}\\
        & \soc_i^t = \soc_i^{t-1} + \frac{(P_{ch_i}^t-P_{dch_i}^t)}{E_i} \Delta T & \text{for every } i\in \Ncal_S, t\in T,\label{Constr:soc2}\\
        & P_{ch_i}^t, P_{dch_i}^t \geq 0 & \text{for every } i\in \Ncal_S, t\in T, \label{Constr: P_ch_P_dch_nonneg}\\
        & \mu_{ch_i}^t, \mu_{dch_i}^t \in \{0,1\} & \text{for every } i\in \Ncal_S, t\in T,\\
        & \mu_{ch_i}^t P_{ch_i,\text{min}}^t \leq P_{ch_i}^t \leq \mu_{ch_i}^t P_{ch_i,\text{max}}^t & \text{for every } i\in \Ncal_S, t\in T, \label{Constr: P_charge_bound}\\
        & \mu_{dch_i}^t P_{dch_i,\text{min}}^t \leq P_{dch_i}^t \leq \mu_{dch_i}^t P_{dch_i,\text{max}}^t & \text{for every } i\in \Ncal_S, t\in T,\label{Constr: P_discharge_bound}\\
        & \mu_{ch_i}^t+\mu_{dch_i}^t \leq 1 & \text{for every } i\in \Ncal_S, t\in T.\label{Constr: mu_bound}
        \end{align}
    \end{subequations}
\end{enumerate}
Constraints \eqref{Constr:soc1} and \eqref{Constr:soc2} set the upper/lower bounds of soc and give the relationships between soc and charging/discharging actions. Constraints \eqref{Constr: P_ch_P_dch_nonneg} -- \eqref{Constr: mu_bound} depict the connection between storage actions $\mu_{ch}$ and $\mu_{dch}$ and their corresponding real power output $P_{ch_i}^t,P_{dch_i}^t$.

Lastly, the SO's cost function is given by 
$$c^\top y\coloneqq \sum_{t\in T} \sum_{i\in \mathcal{N}_G}(r_i^{+,t}P_{G_i}^{+,t} +r_i^{-,t}P_{G_i}^{-,t}) + p_{sl}^t(P_{sl}^t-P_{fl}^t) + \sum_{i\in\mathcal{N}_{DG}} (f_i^+ P_{DG_i}^{t,+}+f_i^- P_{DG_i}^{t,-}).$$

The cost function aims to minimize the electricity cost and reduce the deviation between the intra-day system operation strategy and the day-ahead planning. Thus, the whole adjustment can be summarized as solving
$$\min_{y\in \Ycal(x,h)} c^\top y.$$
Following the structure from Section \ref{Sec: problem setting}, we denote by $Ay\geq b$ the constraints, that solely deal with third-level variables, i.e., every constraint of \eqref{Constr: adjust_fossil_fuel_generators}-\eqref{Constr: mu_bound} despite of \eqref{Constr: P_DG_leq_P_DG_max}. We observe, that the only remaining constraint is the upper bound on the second-level (adverserial) variables $h=P_{DG,\max}$ given by \eqref{Constr: P_DG_leq_P_DG_max} implying $B = \begin{pmatrix}
    0 & -I_{DG}
\end{pmatrix}, B_x=0, B_h = -I_{DG,\max}, b_0=0$. Consequently, the third-level program in this particular case reads
\begin{subequations}\label{Prob: third_level_primal}
    \begin{align}
        \min\ & c^\top y \\
        \st\ & Ay \geq b,\\
        & -P_{DG_i}^t \geq - P_{DG_i,\max}^t && \text{ for every } i \in \Ncal_{DG}, t\in T.\label{Constr: P_DG_ub}
    \end{align}
\end{subequations}
In addition, we denote by $a_{DG}$ the columns of $A$ corresponding to $P_{DG}$ and the remaining columns by $A'$, i.e., $A=\begin{bmatrix}A' & a_{DG}\end{bmatrix}$. Let further $\alpha$ denote the dual variables that correspond to Constraints \eqref{Constr: adjust_fossil_fuel_generators} -- \eqref{Constr: P_ch_P_dch_nonneg}, i.e. to $Ay\geq b$ and $\beta \in \R^{\Ncal_{DG}\times T}$ denote the dual variables corresponding to \eqref{Constr: P_DG_ub}. Then the dual program of \eqref{Prob: third_level_primal} is
\begin{subequations}\label{Prob: third_level_dual}
    \begin{align}
        \max\ & b^\top \alpha - \sum_{i\in \Ncal_{DG}, t\in T} P_{DG_i,\max}^t \beta_{DG_i,t} \\
        \st\ & (A')^\top \alpha= c,\\
        & a_{DG_i}^\top \alpha - \beta_{DG_i,t} = 0 && \text{ for every } i\in \Ncal_{DG}, t\in T,\\
        & \alpha,\beta\geq 0,
    \end{align}
\end{subequations}
We observe, that, as in Section \ref{Sec: problem setting}, on the second level the objective decomposes into a bilinear part and a linear one. However, one can argue that the dual variables $\beta_{DG}$ as well as the maximal capacity $P_{DG,\max}$ of the distributed generators are bounded, i.e., Assumption \ref{Ass: bounds_on_h_and_beta} holds. In particular, we will show that there are $P_{i,t}^+,\beta_{i,t}^+$ such that
$$P_{DG_i,\min}^t \leq P_{DG_i,\max}^t\leq P_{i,t}^+ \text{ and } \beta_{i,t}^-\leq \beta_{DG_i,t} \leq \beta_{i,t}^+.$$
On the one hand, this is because distributed generators have technical limits. For instance, the power outputs of wind turbines are restricted by the cut-out wind speed. The outputs will not exceed the power corresponding to this wind speed. As for the solar panels, their outputs are also restricted by the rated power of the devices themselves. Hence, $P_{DG_i,\max}^t$ is always bounded. 

On the other hand, we may prove the existence of an upper bound for $\beta_{DG_i,t}$ and thereby verify Assumption \ref{Ass: bounds_on_h_and_beta} in our application as follows:
\begin{lemma}\label{Lemma: betabounds_powerflow}
    For every optimal solution $(\alpha^*,\beta^*)$ to \eqref{Prob: third_level_dual}, we have that $$p_{sl}^t-f_i^+\leq \beta_{DG_i,t}^*\leq \max\{r_i^{+,t},r_i^{-,t}, p_{sl}^t\}-\min\{f_i^+,f_i^-\} \text{ for every  }i\in \Ncal_{DG}, t\in T.$$
\end{lemma}

\begin{proof}
    Consider the following relaxed version of \eqref{Prob: third_level_primal}, where we penalized violations in \eqref{Constr: P_DG_ub} instead of incorporating \eqref{Constr: P_DG_ub} as a hard constraint:
    \begin{subequations}\label{Prob: penalized_third_level_primal}
    \begin{align}
        \min\ & c^\top y + \sum_{i\in \Ncal_{DG},t\in T} \beta_{i,t}^+ \lambda_{i,t}\\
        \st\ & Ay\geq b,\\
        & -P_{DG_i}^t + \lambda_{i,t} \geq -P_{DG_i,\max}^t && \text{for every } i\in \Ncal_{DG}, t\in T\\
        & \lambda_{i,t} \geq 0 && \text{for every } i\in \Ncal_{DG}, t\in T,
    \end{align}
\end{subequations}
where $\beta_{i,t}^+\coloneqq \max\{r_i^{+,t},r_i^{-,t}, p_{sl}^t\}-\min\{f_i^+,f_i^-\}$.
Suppose $\lambda_{i,t}>0$, then we can increase the value of $P_{DG_i}^t$ by at most $\lambda_{i,t}$, thereby at least decreasing the objective value by $\min\{f_i^+,f_i^-\}\lambda_{i,t}$. Due to the power balance equations \eqref{Constr: power_flow1} and \eqref{Constr: power_flow2}, we have to either sell the energy on the (second-level) market, i.e., decrease $P_{sl}^t$, which results in a benefit of $p_{sl}^t\lambda_{i,t}$, (in-)decrease $P_{ch_i},P_{dch_i}^t,P_{G_i,\text{reg}}$ resulting either in a benefit of $0$ or at most $\max\{r_i^{+,t},r_i^{-,t}\}\lambda_{i,t}$ respectively. Given this $\beta_{i,t}^+$, we obtain by strong duality:

\begin{subequations}\label{Prob: penalized_third_level_problem_dual}
    \begin{align}
        c^\top y^*=\max\ & b^\top \alpha - \sum_{i\in \Ncal_{DG}, t\in T} P_{DG_i,\max}^t \beta_{DG_i,t}\\
        \st\ & (A')^\top \alpha= c,\\
        & a_{DG_i}^\top \alpha - \beta_{DG_i,t} = 0 && \text{ for every } i\in \Ncal_{DG}, t\in T,\\
        & \beta_{DG_i,t} + \gamma_{i,t} = \beta_{i,t}^+ && \text{ for every } i\in \Ncal_{DG}, t\in T,\\
        & \alpha,\beta,\gamma \geq 0. \label{Constr: lemma_bound_beta_last}
    \end{align}
\end{subequations}
Here, the last two constraints imply
$$\beta_{DG_i,t}\leq \beta_{i,t}^+ \text { for every } i\in \Ncal_{DG}, t\in T,$$
i.e. adding a sufficiently large upper bound on $\beta_{DG}$ does not change the outcome of the dual program and hence, we can safely assume $0 \leq \beta_{DG} \leq \beta^+$.

For the inequality $p_{sl}^t-f_i^+\leq \beta_{DG_i,t}^*$, we observe that \eqref{Prob: penalized_third_level_primal} is unbounded whenever $\beta_{i,t}^+<p_{sl}^t-f_i^+$:
Consider an optimal solution $(y^*,P_{DG}^*)$ of \eqref{Prob: third_level_primal}, which is feasible for \eqref{Prob: penalized_third_level_primal} with $\lambda_{i,t}^*=0$ for every $i\in \Ncal_{DG}, t\in T$.
If we fix $k\in \Ncal_{DG}, t'\in T$, then we observe that $P_{DG_k}^{t'} (\mu)\coloneqq (P_{DG_k}^{t'})^* + \mu$, $P_{DG_k}^{+,t'} (\mu)\coloneqq (P_{DG_k}^{+,t'})^* + \mu$, $\lambda_{k,t'}(\mu) \coloneqq \lambda_{k,t'}^* +\mu$, $P_{sl}^{t'}(\mu)\coloneqq (P_{sl}^{t'})^* -\mu$ is also feasible for \eqref{Prob: penalized_third_level_primal} for every $\mu >0$.
Its objective value is 
$$c^\top y^* - p_{sl}^{t'} \mu + f_k^+ \mu + \sum_{i\in \Ncal_{DG}, t\in T} \beta_{i,t}^+ \lambda_{i,t}^* + \beta_{k,t'}^+ \mu,$$
which tends to $-\infty$ for $\mu\rightarrow \infty$, if $\beta_{k,t'}^+<p_{sl}^{t'}-f_k^+$. By strong duality, we obtain, that for $\beta_{k,t'}^+<p_{sl}^{t'}-f_k^+$ \eqref{Prob: penalized_third_level_problem_dual} does not have a feasible solution. Since $k,t'$ have been chosen arbitrarily, the following problem is equivalent to \eqref{Prob: penalized_third_level_problem_dual}:

\begin{subequations}
    \begin{align}
        c^\top y^*=\max\ & b^\top \alpha - \sum_{i\in \Ncal_{DG}, t\in T} P_{DG_i,\max}^t \beta_{DG_i,t} \\
        \st\ & (A')^\top \alpha= c,\\
        & a_{DG_i}^\top \alpha - \beta_{DG_i,t} = 0 && \text{ for every } i\in \Ncal_{DG}, t\in T,\\
        & \beta_{DG_i,t} + \gamma_{i,t} = \max\{r_i^{+,t},r_i^{-,t}, p_{sl}^t\}-\min\{f_i^+,f_i^-\} && \text{ for every } i\in \Ncal_{DG}, t\in T,\\
        & \beta_{DG_i,t} - \delta_{i,t} = p_{sl}^t-f_i^+ && \text{ for every } i\in \Ncal_{DG}, t\in T,\\
        & \alpha,\beta,\gamma, \delta \geq 0. 
    \end{align}
\end{subequations}
\end{proof}

Moreover, this bound might even be sharp due to the following observation:
\begin{remark}
Depending on the market operations, that is considered, we may have that both, $r_i^{+,t},r_i^{-,t} \leq p_{sl}^t$ and $f_i^+\leq f_i^-$ holds. In this case the two bounds in Lemma \ref{Lemma: betabounds_powerflow} coincide and
$$\beta_{DG_i,t}^*=p_{sl}^t-f_i^+.$$
\end{remark}

Thus, in those cases the McCormick envelopes are exact and the upcoming corollaries provide an exact reformulation of \eqref{Prob: entire_adjustable_RO}. However, even if the McCormick envelope is not exact, the boundedness of both, $\beta_{DG}$ and $P_{DG,\max}$ enables us to relax the second level problem $\max_{h\in \Omega} D(x,h)$ with the McCormick envelope. This gives rise to the following corollary of Theorem \ref{Thm: linear_single_level_approximation_general}:
\begin{coroll}\label{Cor: linear_single_level_approximation}
    Let $\Omega$ be a robust ambiguity set as in \eqref{Eq: Def_Omega}, where $\Omega$ denotes a polytope with potential nonnegative slack variables $\eta$ in the rows $i\in I$, i.e. $\Omega=\left\{P_{DG,\max}\in \R^{\Ncal_{DG}\times T}, \eta \in \R^I_{\geq 0}: A_\Omega^\top P_{DG,\max} + \eta =b_\Omega\right\}$. Let further $\beta_{i,t}^-,\beta_{i,t}^+$ be a lower/upper bound for $\beta_{DG_i,t}$. Then, the following linear program provides an upper bound to \eqref{Prob: entire_adjustable_RO}:
    \begin{subequations}
    \begin{align}
        \min\ & \sum_{i\in \Ncal_G,t\in T} c_{G_i^t,2} (P_{G_i}^t)^2 + c_{G_i^t,1} P_{G_i}^t + c_{G_i^t,0} + \sum_{t\in T} p_{fl}^t P_{fl}^t + c^\top y\notag\\
        & + \sum_{i\in \Ncal_{DG}, t\in T}\beta_{i,t}^+u_{i,t}^{\beta,+}
        + \sum_{i\in \Ncal_{DG}, t\in T}\beta_{i,t}^-u_{i,t}^{\beta,-}
        + b_\Omega^\top u_\Omega \notag\\
        & - \sum_{i\in \Ncal_{DG}, t\in T}P_{DG_i,\min}^t \beta_{i,t}^- \env_{i,t}^1 - \sum_{i\in \Ncal_{DG}, t\in T} P_{i,t}^+\beta_{i,t}^+ \env_{i,t}^2 \notag\\
        & - \sum_{i\in \Ncal_{DG}, t\in T} P_{i,t}^+\beta_{i,t}^- \env_{i,t}^3 - \sum_{i\in \Ncal_{DG}, t\in T} P_{DG_i,\min}^t\beta_{i,t}^+ \env_{i,t}^4\\
        \st\ & Ay \geq b,\\
        & -P_{DG_i}^t + u_{i,t}^{\beta,+} + u_{i,t}^{\beta,-} -P_{DG_i,\min}^t\env_{i,t}^1 - P_{i,t}^+\env_{i,t}^2\notag\\
        &\qquad - P_{i,t}^+ \env_{i,t}^3 -P_{DG_i,\min}^t\env_{i,t}^4 \geq 0 && \forall i\in \Ncal_{DG}, t\in T\\
        & u_{i,t}^{\beta,+} \geq 0,\ -u_{i,t}^{\beta,-} \geq 0 && \forall i\in \Ncal_{DG}, t\in T\\
        & (A_\Omega u_\Omega)_i -\beta_{i,t}^- (\env_{i,t}^1 + \env_{i,t}^3) -\beta_{i,t}^+ (\env_{i,t}^2 + \env_{i,t}^4) \geq 0 && \forall i\in \Ncal_{DG}, t\in T\\
        & u_\Omega \geq 0\\
        & \env_{i,t}^1 + \env_{i,t}^2 + \env_{i,t}^3 + \env_{i,t}^4 \geq -1 && \forall i\in \Ncal_{DG}, t\in T\\
        & -\env_{i,t}^1, -\env_{i,t}^2 \geq 0 && \forall i\in \Ncal_{DG}, t\in T\\
        & \env_{i,t}^3, \env_{i,t}^4 \geq 0 && \forall i\in \Ncal_{DG}, t\in T\\
        & \eqref{Constr: fl1_define_Pfl}, \ \& \ \eqref{Constr: fl2_bounds_P_G}
    \end{align}
\end{subequations}
\end{coroll}

We note, that Corollary \ref{Cor: linear_single_level_approximation} is a direct consequence of Theorem \ref{Thm: linear_single_level_approximation_general} and thus its proof follows the same lines. However, we include the full proof here, as the notation varies a bit and the proof illustrates the impact of Lemma \ref{Lemma: betabounds_powerflow}.

\begin{proof}
    We observe that with \eqref{Prob: third_level_dual} and Lemma \ref{Lemma: betabounds_powerflow} the second-level $\max_{P_{DG,\max}\in \Omega} \min_{y\in \Ycal(x,h)} c^\top y$ can be written as
\begin{subequations}\label{Prob: second_level_primal}
    \begin{align}
        \max\ & b^\top \alpha - \sum_{i\in \Ncal_{DG}, t\in T} P_{DG_i,\max}^t \beta_{DG_i,t} \\
        \st\ & (A')^\top \alpha= c,\\
        & a_{DG_i}^\top \alpha - \beta_{DG_i,t} = 0 && \text{ for every } i\in \Ncal_{DG}, t\in T,\\
        & \beta_{DG_i,t} + \gamma_{i,t} = \beta_{i,t}^+ && \text{ for every } i\in \Ncal_{DG}, t\in T,\\
        & \beta_{DG_i,t} - \delta_{i,t} = \beta_{i,t}^- && \text{ for every } i\in \Ncal_{DG}, t\in T,\\
        & A_\Omega^\top P_{DG,\max} + \eta = b_\Omega,\\
        & \alpha,\beta,\gamma,P_{DG,\max},\delta, \eta \geq 0,
    \end{align}
\end{subequations}
where $\beta^-,\beta^+$ are chosen as in Lemma \ref{Lemma: betabounds_powerflow}. Next, we substitute $\kappa_{i,t}\coloneqq P_{DG_i,\max}^t \beta_{DG_i,t}$ in the objective and relax the resulting constraint $\kappa_{i,t}\coloneqq P_{DG_i,\max}^t \beta_{DG_i,t}$ by a McCormick envelope. Note, that since $P_{DG_i,\max}^t, \beta_{DG_i,t}\geq 0$, we can immediately conclude $\kappa\geq 0$, which simplifies our notation a bit. If we further introduce suitable nonnegative slack variables, we obtain the following dual LP: 
\begin{subequations}\label{Prob: second_level_primal_relaxed}
    \begin{align}
        \max\ & b^\top \alpha - \sum_{i\in \Ncal_{DG}, t\in T} P_{DG_i,\max}^t \beta_{DG_i,t} \label{Constr: Thm1_dual_first}\\
        \st\ & (A')^\top \alpha= c,\\
        & a_{DG_i}^\top \alpha - \beta_{DG_i,t} = 0 && \text{ for every } i\in \Ncal_{DG}, t\in T,\\
        & \beta_{DG_i,t} + \gamma_{i,t} = \beta_{i,t}^+ && \text{ for every } i\in \Ncal_{DG}, t\in T,\\
        & \beta_{DG_i,t} - \delta_{i,t} = \beta_{i,t}^- && \text{ for every } i\in \Ncal_{DG}, t\in T,\\
        & A_\Omega^\top P_{DG,\max} + \eta = b_\Omega,\\
        & \kappa_{i,t} = P_{DG_i,\min}^t \beta_{DG_i,t} + P_{DG_i,\max}^t\beta_{i,t}^- - P_{DG_i,\min}^t \beta_{i,t}^- + \eta_{i,t}^1  && \text{for every } i\in \Ncal_{DG}, t\in T, \\
        & \kappa_{i,t} = P_{i,t}^+\beta_{DG_i,t} + P_{DG_i,\max}^t \beta_{i,t}^+ - P_{i,t}^+\beta_{i,t}^+ + \eta_{i,t}^2 = 0  && \text{for every } i\in \Ncal_{DG}, t\in T, \label{Constr: McCormick_lb2}\\
        & \kappa_{i,t} = P_{i,t}^+ \beta_{DG_i,t} + P_{DG_i,\max}^t \beta_{i,t}^- -P_{i,t}^+\beta_{i,t}^- -\eta_{i,t}^3 && \text{for every } i\in \Ncal_{DG}, t\in T, \\%\label{Constr: McCormick_ub1}\\
         & \kappa_{i,t} = P_{DG_i,\max}^t \beta_{i,t}^+ + P_{DG_i,\min} \beta_{DG_i,t} - P_{DG_i,\min}^t\beta_{i,t}^+ -\eta_{i,t}^4  && \text{for every } i\in \Ncal_{DG}, t\in T, \\%\label{Constr: McCormick_ub2}\\
        & \alpha,\beta,\gamma, \delta,P_{DG,\max},\rho,\eta,\kappa \geq 0.
    \end{align}
\end{subequations}

If we denote the dual variables of \eqref{Constr: Thm1_dual_first} -- \eqref{Constr: McCormick_lb2} by $y, u_{i,t}^{\beta,+}, u_{i,t}^{\beta,-}, u_\Omega, \env_{i,t}^1,\env_{i,t}^2,\env_{i,t}^3,\env_{i,t}^4$ respectively, then the result follows by strong duality and including the first-level variables and objectives.
\end{proof}
Again, we observe that Corollary \ref{Cor: linear_single_level_approximation} provides an LP inner approximation of the linear relaxation of \eqref{Prob: entire_adjustable_DRO}. Now, the following direct corollary of Theorem \ref{Thm: single_level_MIP_general} incorporates the discrete (binary) decisions $\mu_{ch}$ and $\mu_{dch}$. 

\begin{coroll}\label{Cor: single_level_MIP}
    Let $\Omega$ be a robust ambiguity set as in \eqref{Eq: Def_Omega}, where $\Omega$ denotes a polytope with potential nonnegative slack variables $\eta$ in the rows $i\in I$, i.e. $\Omega=\left\{P_{DG,\max}\in \R^{\Ncal_{DG}\times T}, \eta \in \R^I_{\geq 0}: A_\Omega^\top P_{DG,\max} + \eta =b_\Omega\right\}$. Let further $\beta_{i,t}^-,\beta_{i,t}^+$ be a lower/upper bound for $\beta_{DG_i,t}$. Then, the following MIP provides an upper bound to the tri-level MIP \eqref{Prob: entire_adjustable_DRO} with the given parameters from Section \ref{Sec: powerflow_theory}:
    \begin{subequations}\label{Prob: single_level_MIP}
    \begin{align}
         \min\ & \sum_{i\in \Ncal_G,t\in T} c_{G_i^t,2} (P_{G_i}^t)^2 + c_{G_i^t,1} P_{G_i}^t + c_{G_i^t,0} + \sum_{t\in T} p_{fl}^t P_{fl}^t + c^\top y\notag\\
        & + \sum_{i\in \Ncal_{DG}, t\in T}\beta_{i,t}^+u_{i,t}^{\beta,+}
        + \sum_{i\in \Ncal_{DG}, t\in T}\beta_{i,t}^-u_{i,t}^{\beta,-}
        + b_\Omega^\top u_\Omega \notag\\
        & - \sum_{i\in \Ncal_{DG}, t\in T}P_{DG_i,\min}^t \beta_{i,t}^- \env_{i,t}^1 - \sum_{i\in \Ncal_{DG}, t\in T} P_{i,t}^+\beta_{i,t}^+ \env_{i,t}^2 \notag\\
        & - \sum_{i\in \Ncal_{DG}, t\in T} P_{i,t}^+\beta_{i,t}^- \env_{i,t}^3 - \sum_{i\in \Ncal_{DG}, t\in T} P_{DG_i,\min}^t\beta_{i,t}^+ \env_{i,t}^4\\
        \st\ & Ay \geq b,\\
        & -P_{DG_i}^t + u_{i,t}^{\beta,+} + u_{i,t}^{\beta,-} -P_{DG_i,\min}^t\env_{i,t}^1 - P_{i,t}^+\env_{i,t}^2 \notag\\
        & \qquad - P_{i,t}^+ \env_{i,t}^3 -P_{DG_i,\min}^t\env_{i,t}^4 \geq 0 && \forall i\in \Ncal_{DG}, t\in T\\
        & u_{i,t}^{\beta,+} \geq 0,\ -u_{i,t}^{\beta,-} \geq 0 && \forall i\in \Ncal_{DG}, t\in T\\
        & (A_\Omega u_\Omega)_i -\beta_{i,t}^- (\env_{i,t}^1 + \env_{i,t}^3) -\beta_{i,t}^+ (\env_{i,t}^2 + \env_{i,t}^4) \geq 0 && \forall i\in \Ncal_{DG}, t\in T\\
        & u_\Omega \geq 0\\
        & \env_{i,t}^1 + \env_{i,t}^2 + \env_{i,t}^3 + \env_{i,t}^4 \geq -1 && \forall i\in \Ncal_{DG}, t\in T\\
        & -\env_{i,t}^1, -\env_{i,t}^2 \geq 0 && \forall i\in \Ncal_{DG}, t\in T\\
        & \env_{i,t}^3, \env_{i,t}^4 \geq 0 && \forall i\in \Ncal_{DG}, t\in T\\   
        & \eqref{Constr: fl1_define_Pfl}, \ \& \ \eqref{Constr: fl2_bounds_P_G} \label{Constr: single_level_MIP_env2_lb}\\
        & \mu_{ch}, \mu_{dch} \in \{0,1\}^{\Ncal_S\times T}\label{Constr: IP_constraint} \\
        & P_{G_i}^t, P_{fl}^t, y', P_{DG_i}^t, u^{\beta,+}_{i,t}, u^{\beta,-}_{i,t}, u^P_{i,t}, \env^1_{i,t}, \env^2_{i,t}, \env^3_{i,t}, \env^4_{i,t} \in \R.
    \end{align}
\end{subequations}
\end{coroll}

Thus, our approximation technique is applicable to the optimization of smart converters in power grids. Moreover, the only strict relaxation comes from approximating the uncertainty set $\Omega$ by McCormick envelopes of the bilinear terms and may, depending on the second-level (intra-day) market price $p_{sl}$ and the penalizations for adjustments $f^+,f^-,r^+,r^-$, even be sharp. Hence, it seems natural to test our approximations numerically. 
\section{Computational results}\label{Sec:computationals_results}

We present the results of three case studies in this section. All of these instances are considered on a daily basis divided into hourly (24 period) or 15min (96 period) time intervals. The first benchmark is a 5-bus instance, based on the \enquote{case5.m} instance from the matpower library \cite{Zimmerman2011a}. Second, we consider a 30-bus instance, based on the \enquote{case\_ieee30.m} instance of the matpower library and finally a modified version of the IEEE 118-bus system similar to the one in \cite{Cobos2018a}. The computations were executed via Gurobi 10.0.0 under Python 3.7 on a Macbook Pro (2019) notebook with an Intel Core i7 2,8 GHz Quad-core and 16 GB of RAM. 

\subsection{5-bus example}
The topology of the test system is shown in Figure \ref{Fig: 5-bus_system}, where we consider bus $1$ to be the root node. We observe, that the conventional generators are connected to the buses $1$ and $4$, i.e. $\Ncal_G=\{1,4\}$, two distributed generators are connected to buses $1$ and $5$, i.e. $\Ncal_{DG}=\{1,5\}$ and an energy storage unit is connected to bus $3$, i.e. $\Ncal_S=\{3\}$. 

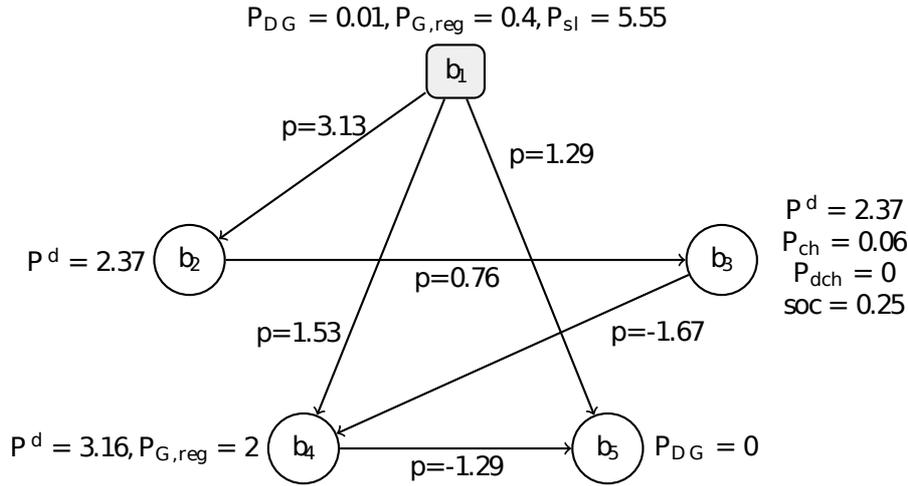
\begin{figure}[ht]\label{Fig: 5-bus_system}

\tikzstyle{block} = [rectangle, draw=black, fill=lightgray!25,text width=1.5em, rounded corners, minimum height=2em, text centered]

\tikzstyle{block3}=[circle, draw=black, fill=white!10,text width=1.5em,rounded corners, minimum height=2em, text centered]

\tikzstyle{line} = [draw, -latex']

\begin{center}
\begin{tikzpicture}[node distance = 2cm, auto,thick]
\node [block] at (0,5)(b1) [label={[align=center]above: $P_{DG}=0.01, P_{G,\text{reg}} =0.4, P_{sl}=5.55$}]{$b_1$};
\node [block3] at (-3.5,2.5) (b2) [label={[align=center]left: $P^d=2.37$}] {$b_2$};
\node [block3] at (3.5,2.5) (b3) [label={[align=center]right: $\begin{array}{c}
						P^d=2.37\\
						P_{ch}=0.06\\
						P_{dch}=0\\
						\text{soc} = 0.25
				\end{array}$}]  {$b_3$};
\node [block3] at (-2,0) (b4) [label={[align=center]left: $P^d=3.16, P_{G,\text{reg}}=2$}] {$b_4$};
\node [block3] at (2,0) (b5) [label={[align=center]right: $P_{DG}=0$}]  {$b_5$};
\draw [->] (b1) to node [left, near start]{p=3.13} (b2);
\draw [->] (b1) to node [left, near end]{p=1.53} (b4);
\draw [->] (b1) to node [auto, near start]{p=1.29} (b5);
\draw [->] (b2) to node [below, midway]{p=0.76} (b3);
\draw [->] (b3) to node [auto, near start]{p=-1.67} (b4);
\draw [->] (b4) to node [below, midway]{p=-1.29} (b5);
\end{tikzpicture}

\end{center}

\caption[BDD]{The case5.m network with its corresponding generators and flows at 3am. $P^d=0$ at buses $1$ and $5$}
\end{figure}

Whether a generator in \enquote{case5.m} is a conventional/renewable one or a storage was decided by the authors.
Both, the day-ahead and intra-day market prices $p_{fl}, p_{sl}$ were taken as averages from the Pecan street data base's \cite{pecanstreet} \enquote{miso} data set for October 3rd, 2022. The daily deviations in $P^d$, denoted by $\Delta_d^t$, or daily deviations in $P_{DG,\max}^t$, denoted by $\Delta_{DG}^t$ were similarly taken from the Pecan street data base's \enquote{california\_iso} dataset for October 3rd, 2022. 

Then, the demand varying over the day was modeled as $P^d=P^d\cdot \Delta_d^t$ and the varying potential renewable energy production was modeled by $P_{DG,\text{forecast}}^t\coloneqq \min\{\frac{P_{DG,\min}+P_{DG,\max} }{2}\cdot \Delta_{DG}^t, P_{i,t}^+\}$.

Solving Problem \eqref{Prob: single_level_MIP} takes less than 1s. However, its objective value highly depends on the uncertainty imposed on the system. We illustrate in Figure \ref{Fig: 5-bus_system_sensitivity} how sensitive the optimal solution reacts to changes in the maximal forecast error $R$, that crucially determines $\Omega$ through \eqref{Eq: Def_Omega}. In particular, since in \enquote{case5.m} we have neglectable upward and downward regulation costs and $f^-\geq f^+$, Lemma \ref{Lemma: betabounds_powerflow} implies that the McCormick relaxation is sharp if and only if $r^-,r^+ \leq p_{sl}$. As realistic penalties $r^+ = r^-$, we assume $r_1^+=r_1^-= 14$\$ p.u., which are the costs of operating the first generator at bus $1$. For the sake of a better analysis, we replace the natural choice $r_4^+=r_4^-= 40$\$ p.u., which are the costs of operating the generator at bus $4$ by $r_4^+=r_4^-= 20$\$ p.u. as then $r^-,r^+ \leq p_{sl}$ and we can compare variations in the penalties to an optimal robust solution.

In particular, for this instance, $R=1$ yields the nominal optimal solution with an objective value of $727082\$$.

\begin{figure}[ht]\label{Fig: 5-bus_system_sensitivity}
\begin{tikzpicture}
\begin{axis}[
    xlabel={$R$},
    ylabel={Objective value [1000\$]},
    xmin=0, xmax=1,
    ymin=700, ymax=1300,
    xtick={0,0.2,0.4,0.6,0.8,1},
    ytick={700,900,1100,1300},
    legend pos=north east,
    ymajorgrids=true,
    grid style=dashed,
]

\addplot[
    color=blue,
    mark=square,
    ]
    coordinates {
    (0,1178.39)(0.25,1065.56)(0.5,952.737)(0.75,839.909)(1,727.082)
    };
    \addlegendentry{Objective per $R$ with $r^+_4=r^-_4= 20$}

\addplot[
    color=red,
    mark=square,
    ]
    coordinates {
    (0,1192.93)(0.25,1080.1)(0.5,967.27)(0.75,854.442)(1,741.615)
    };
    \addlegendentry{Objective per $R$ with $r^+_4=r^-_4= 40$}
    
\end{axis}
\end{tikzpicture}

\caption[BDD]{Objective value of \eqref{Prob: single_level_MIP} on \enquote{case5.m} under varying $R$}
\end{figure}
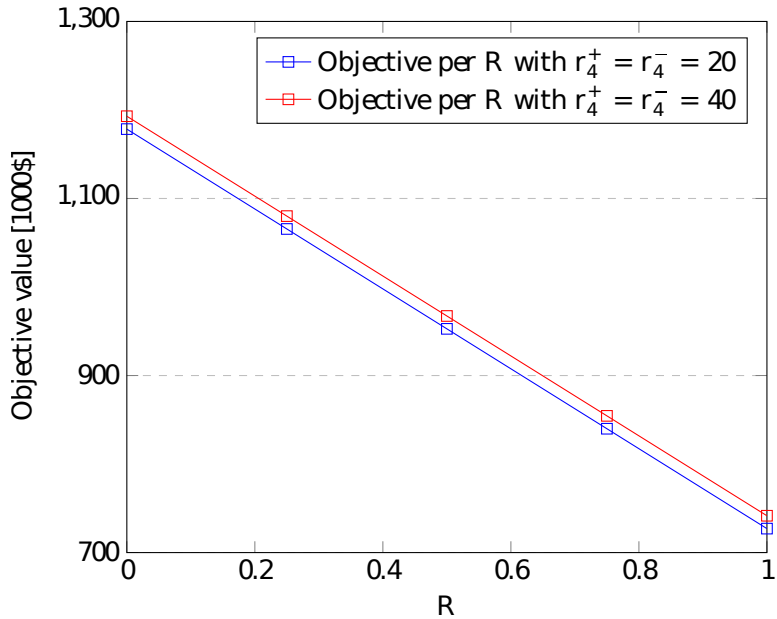
Moreover, the blue line in Figure \ref{Fig: 5-bus_system_sensitivity} illustrates the perfect linear relation between the lower bound of the uncertainty set $\Omega$ and the objective value of \eqref{Prob: single_level_MIP}. We conclude, that the DSO may cut its worst-case costs by almost $50\%$ with a perfectly accurate weather forecast, i.e. $\Omega=\{P_{DG,forecast}\}$. As this is unrealistic with present forecasting methods, we would like to highlight that one may gain already significant cost reductions in the worst-case by incorporating more information on $\Omega$. 

Moreover, the red line in Figure \ref{Fig: 5-bus_system_sensitivity} shows an upper bound to \eqref{Prob: entire_adjustable_RO} given by Corollary \ref{Cor: single_level_MIP} in case $r^{+,t}_4 > p_{sl}^t$ for some $t$, i.e. in case Lemma \ref{Lemma: betabounds_powerflow} is violated. Since the objective value with respect to this penalty is contained between the red and the blue line, Figure \ref{Fig: 5-bus_system_sensitivity} thereby shows a rather strong approximation quality for this particular instance. However, we want to stress, that this only holds for this particular example and may not be a general pattern.

\subsection{30-bus example}
Similarly, as in the 5-bus example, the topology of the 30-bus test system is taken from \enquote{case\_ieee30.m}, a system with 41 transmission lines and after modification 4 dispatchable generators as well as 4 energy storages. The only renewable generator is placed at bus $2$, i.e. $\Ncal_{DG}=\{2\}$. In addition, we choose $\Ncal_{G}=\{5,8,11,13\}$ and the four energy storage units to be connected to buses $1,2,8,13$, i.e. $\Ncal_S=\{1,2,8,13\}$. The estimation procedure of market prices and demands are kept from the \enquote{case5.m} example. Since also \enquote{case\_ieee30.m} does not include upward or downward regulation costs for the generators, the McCormick envelope is sharp and for $R=1$, i.e. $\Omega=\{P_{DG,forecast}\}$, the nominal value of $104,088\$$ is attained.

Moreover, we would like to illustrate the dependency of the worst-case revenue with respect to choices of $R$ in Figure \ref{Fig: 30-bus_system_sensitivity}. We want to highlight, that also in this more elaborate example, the runtime was $<1$s.

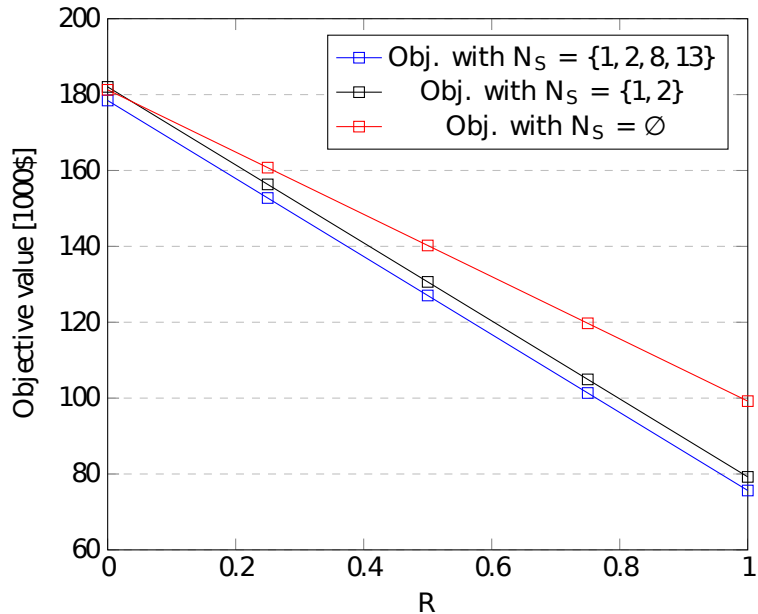
\begin{figure}[ht]\label{Fig: 30-bus_system_sensitivity}
\begin{tikzpicture}
\begin{axis}[
    xlabel={$R$},
    ylabel={Objective value [1000\$]},
    xmin=0, xmax=1,
    ymin=60, ymax=200,
    xtick={0,0.2,0.4,0.6,0.8,1},
    ytick={60,80,...,200},
    legend pos=north east,
    ymajorgrids=true,
    grid style=dashed,
]

\addplot[
    color=blue,
    mark=square,
    ]
    coordinates {
    (0,178.445)(0.25,152.747)(0.5,127.049)(0.75,101.351)(1,75.652)
    };
    \addlegendentry{Obj. with $\Ncal_S=\{1,2,8,13\}$}

\addplot[
    color=black,
    mark=square,
    ]
    coordinates {
    (0,182.030)(0.25,156.332)(0.5,130.634)(0.75,104.935)(1,79.237)
    };
    \addlegendentry{Obj. with $\Ncal_S=\{1,2\}$}

\addplot[
    color=red,
    mark=square,
    ]
    coordinates {
    (0,181.263)(0.25,160.749)(0.5,140.235)(0.75,119.720)(1,99.206)
    };
    \addlegendentry{Obj. with $\Ncal_S=\emptyset$}
    
\end{axis}
\end{tikzpicture}

\caption[BDD]{Objective values of \eqref{Prob: single_level_MIP} with $|\Ncal_S|\in \{0,2,4\}$ on \enquote{case\_ieee30.m}}
\end{figure}
Note, that different slopes may occur due to the different capabilities of the storages. In particular, the improved performance from $\Ncal_S=\emptyset$ to $\Ncal_S\neq \emptyset$ indicates that the capability of storing all renewable energy produced within the transmission system is more valuable than storing energy from prior purchases, i.e. externally produced energy.

\subsection{Analyzing the runtime on larger instances}
After illustrating the behavior of the objective value under different uncertainties and storages, we focus on the main advantage of the proposed MIP approach, namely its speed. As the scaling of the runtime is crucial in industrial applications, we demonstrate the applicability of our algorithm to larger power systems, particularly a 118-bus (\enquote{case118.m}), a 200-bus (\enquote{case200.m}) and 300-bus (\enquote{case300.m}) test system. 

To this end, we aim to keep the considered instances as comparable as we can. In particular, we again keep the estimation procedure of market prices and demands from the \enquote{case5.m} example. Additionally, none of the considered instances
contains upward or downward regulation costs for the generators and thereby due to Lemma \ref{Lemma: betabounds_powerflow} we always compute the exact solutions to \eqref{Prob: entire_adjustable_RO}. Lastly, we note that the number of storages $|\Ncal_S|$ determines the amount of binary variables in \eqref{Prob: single_level_MIP} and consequently should crucially impacts the runtime. Thus, we equipped our test systems with $|\Ncal_S|=6$, $|\Ncal_S|=10$ and $|\Ncal_S|=15$ storages respectively in order to achieve an approximately constant ratio of buses to storages $|V|/|\Ncal_S|$, i.e. $|V|/|\Ncal_S|\approx 20$. To further improve comparability, we also recomputed the $30$-bus system with $2$ storages instead of $4$. The following figure illustrates the achieved runtime:

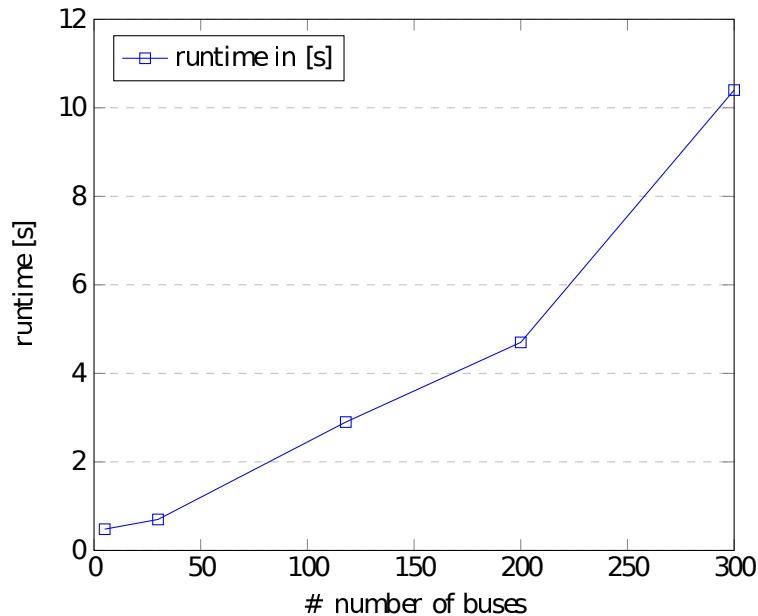
\begin{figure}[ht]\label{Fig: runtime}
\begin{tikzpicture}
\begin{axis}[
    xlabel={\# number of buses},
    ylabel={runtime [s]},
    xmin=0, xmax=300,
    ymin=0, ymax=12,
    xtick={0,50,...,300},
    ytick={0,2,...,12},
    legend pos=north west,
    ymajorgrids=true,
    grid style=dashed,
]

\addplot[
    color=blue,
    mark=square,
    ]
    coordinates {
    (5,0.48)(30,0.7)(118,2.9)(200,4.7)(300,10.4)
    };
    \legend{runtime in [s]}
\end{axis}
\end{tikzpicture}

\caption[BDD]{runtime comparison \enquote{case5.m} with $|\Ncal_S|=1$, \enquote{case\_ieee30.m} with $|\Ncal_S|=2$, \enquote{case118.m} with $|\Ncal_S|=6$,
\enquote{case200.m} with $|\Ncal_S|=10$,
\enquote{case300.m} with $|\Ncal_S|=15$}
\end{figure}

As \eqref{Prob: entire_adjustable_RO} is a notoriously challenging problem, see Question in \cite{Conejo2022a} and Section 6 in \cite{Yanikoglu2019a} benchmarks on the exact problem setting are, to the best of our knowledge, rare. However, Cobos et. al. \cite{Cobos2018a} applied a nested column generation approach in order to solve a strongly related problem on \enquote{case118.m} and achieved runtimes between $200s$ and $800s$. The considered instances in \cite{Cobos2018a} contain $|\Ncal_G||T| + 3|\Ncal_S| |T|$ first-level, $2|\Ncal_{DG}| |T|$ second-level and $|\Ncal_S| |T|$ third-level binary variables -- a significantly larger amount of binary variables compared to our instances since in \cite{Cobos2018a} we have $|\Ncal_G|= 54, |\Ncal_{DG}|=10, |\Ncal_S| = 6$ and $|T|=24$. In summary, the nested column generation in \cite{Cobos2018a} addresses $54\cdot 24 + 3\cdot 6 \cdot 24 + 2 \cdot 10 \cdot 24 + 6 \cdot 24 = 98 \cdot 24 = 2352$ binary decisions on a comparable instance. As Figure \ref{Fig: runtime} illustrates, the proposed algorithm solves the instance in $\approx 3$s, but with significantly fewer, namely $2\cdot 6\cdot 24 = 12 \cdot 24 = 288$ binary decisions. Thus, a direct comparison with \cite{Cobos2018a} seems rather inappropriate. 

Furthermore, the parametric programming approach presented in \cite{Avraamidou2019a} can be used to solve the ARO \eqref{Prob: entire_adjustable_RO}, even if \eqref{Prob: entire_adjustable_RO} is not weakly-connected. The same authors demonstrate, that an instance with $60$ binary variables on various levels can be solved within $15$s, see Table 7, Problem P5 in \cite{Avraamidou2020a}. As the number of binary variables of this instance is still significantly smaller than \enquote{case\_ieee30.m}, which the presented MIP framework can solve within $<1$s, it seems natural to conjecture, that the MIP approach outperforms the parametric programming approach on weakly-connected AROs in terms of runtime. However, we would like to highlight, that the work in \cite{Avraamidou2019a} rather aims at wide applicability as the authors present a significantly more general approach.

%\subjclass[2010]{\input{msc2010}}

\section{Conclusion}
This article presents a new MIP framework to approximate adjustable robust programs with integer variables in the innermost (adjustment) stage. It is based on a weak connection between the separate stages and uses a McCormick envelope to strengthen the adversarial, thereby relaxing the ARO. We have proven that the resulting MIP provides feasible solutions for the first stage, that can be adjusted to a solution satisfying an objective at least as good as the ARO objective regardless of the realization of the uncertainty. In addition, we have provided a sufficient criterion, for the exactness of our approximation.

Moreover, we applied our results to model discrete adjustments of smart converters in a power system and provided numerical evidence, that our approach is competitive to previous methods such as the nested column generation or parametric programming in terms of runtime.
\section*{Acknowledgments}
\label{sec:acknowledgements}
We are grateful to Robert Burlacu and Qianwen Xu for stimulating discussions. The authors gratefully acknowledge Digital Futures for financially supporting the project.% the Swedish Government for their support within the project \enquote{Autonomous coordination and control of smart converters for sustainable power systems} in the Strategic Research area \enquote{Digital Futures}.

%\section*{Declarations of interest} None
%%% Local Variables:
%%% mode: latex
%%% TeX-master: "adjustable-robust-lcps-preprint"
%%% End:

\printbibliography

\end{document}